\documentclass[11pt,reqno]{amsart}

\usepackage{amsfonts, amsthm, amsmath}
\allowdisplaybreaks[4]

\usepackage{rotating}

\usepackage{tikz}

\usepackage{graphics}

\usepackage{amssymb}

\usepackage{amscd}

\usepackage[latin2]{inputenc}

\usepackage{t1enc}

\usepackage[mathscr]{eucal}

\usepackage{indentfirst}

\usepackage{graphicx}

\usepackage{graphics}

\usepackage{pict2e}

\usepackage{mathrsfs}

\usepackage{enumerate}
\usepackage{ferrers} 
\usepackage[pagebackref]{hyperref}
\hypersetup{colorlinks=true}
\usepackage{cite}
\usepackage{color}
\usepackage{epic}
\usepackage{hyperref} 
\usepackage{framed}
\usepackage{mathabx}
\usepackage{enumitem}

\numberwithin{equation}{section}
\theoremstyle{plain}

\newtheorem{theorem}{Theorem}[section]

\newtheorem{lemma}[theorem]{Lemma}

\theoremstyle{definition}

\newtheorem{example}[theorem]{Example}

\newtheorem{remark}[theorem]{Remark}

\newtheorem{?}[theorem]{Problem}

\newtheoremstyle{named}{}{}{\itshape}{}{\bfseries}{.}{.5em}{#1\thmnote{ #3}}
\theoremstyle{named}

\newcommand{\f}[1]{\ifthenelse{\equal{#1}{1}}{(q;q)_\infty}{(q^{#1};q^{#1})_{\infty}}}

\def\cD{\mathscr{D}}
\def\cP{\mathscr{P}}
\def\cA{\mathscr{A}}
\def\cB{\mathscr{B}}

\def\cE{\mathscr{E}}
\def\cRR{\mathscr{R}}
\def\cOD{\mathscr{D}^{\mathrm{o}}}
\def\cED{\mathscr{D}^{\mathrm{e}}}
\def\cF{\mathscr{F}}
\def\bN{\mathbb{N}}
\def\sol{\mathrm{sol}}
\def\I{\mathrm{I}}
\def\II{\mathrm{II}}
\def\III{\mathrm{III}}
\def\IV{\mathrm{IV}}
\def\V{\mathrm{V}}
\def\VI{\mathrm{VI}}
\def\VII{\mathrm{VII}}
\def\f{\mathrm{f}}

\def\wt{\mathrm{wt}}
\def\sfi{\mathrm{sfi}}
\def\L{\mathrm{L}}
\def\R{\mathrm{R}}
\def\cM{\mathscr{M}}

\def\ri{\rightarrow}
\def\ga{\gamma}
\def\xri{\xrightarrow}
\def\la{\lambda}
\def\al{\alpha}
\def\ep{\epsilon}
\def\hla{\hat{\la}}
\def\hmu{\hat{\mu}}

\begin{document}
\title[Sequences of odd length in strict partitions IV]{Sequences of odd length in strict partitions IV: the combinatorics of parameterized Rogers-Ramanujan type identities}


\author[H. Li]{Haijun Li}
\address[Haijun Li]{College of Mathematics and Statistics, Chongqing University, Chongqing 401331, P.R. China}
\email{lihaijun@cqu.edu.cn}

\date{\today}

\begin{abstract}
In the first three papers, we conducted a series of discussions on the statistics of strict partitions and Rogers-Ramanujan partitions, specifically the sequences of odd length (denoted as ``$\sol$'') and its extensions. We established bijections for some Rogers-Ramanujan type identities. This paper will continue that series of work, and first we will use the bijective method to re-establish several parameterized Rogers-Ramanujan type identities, which appeared in the recent work of  Hao-Kuai-Xia and Li-Wang. Moreover, we focus on the work of Chen-Yin and parameterize their main results, where the ``$\sol$'' has evolved.

\end{abstract}


\maketitle

\section{Introduction}\label{sec:intro}

We begin with the notation of q-series. Let $q$ denote a complex number with $|q|<1$. Here and in what follows, we adopt the standard $q$-series notation \cite{GR90}. We let
\begin{align*}
&(a; q)_n=(1-a)(1-aq)\cdots (1-aq^{n-1}),\text{ for }n\geq 1,\ (a; q)_0=1,\\
&(a; q)_{\infty}=\lim_{n\ri\infty}(a; q)_n,\text{ and }(a_1, ..., a_m; q)_n=(a_1; q)_n\cdots (a_m; q)_n.
\end{align*}

Rogers-Ramanujan type identities are certain sum-to-product identities in which the left side is a mixed sum of some $q$-hypergeometric series and the right sides are some infinite products. The study of them is inspired by the two famous Roger-Ramanujan identities:
\begin{align}
&\sum_{n\geq 0}\frac{q^{n^2}}{(q; q)_n}=\frac{1}{(q, q^4; q^5)_{\infty}},\label{id:RR1}\\
&\sum_{n\geq 0}\frac{q^{n^2+n}}{(q; q)_n}=\frac{1}{(q^2, q^3; q^5)_{\infty}}.\label{id:RR2}
\end{align}

These identities \eqref{id:RR1} and \eqref{id:RR2} were first proved by Rogers \cite{rog94} and later rediscovered by Ramanujan \cite{ram14, ram19}, and Schur \cite{sch17} independently reproduced them. Rogers also proved the following identities.
\begin{align}
&\sum_{n\geq 0}\frac{q^{n^2}}{(q^4; q^4)_{n}}=\frac{1}{(-q^2; q^2)_{\infty}(q, q^4; q^5)_{\infty}},\label{id:Rog1}\\
&\sum_{n\geq 0}\frac{q^{n^2+2n}}{(q^4; q^4)_{n}}=\frac{1}{(-q^2; q^2)_{\infty}(q^2, q^3; q^5)_{\infty}}.\label{id:Rog2}
\end{align}

In terms of integer partitions, the two identities \eqref{id:RR1} and \eqref{id:RR2} may be interpreted as follows.
\begin{theorem}
For each integer $n\geq 1$,
\begin{itemize}
\item[(1)] the number of partitions of $n$ into parts congruent to $\pm 1$ modulo $5$ is the same as the number of partitions of $n$ such that every two consecutive parts have difference at least 2.

\item[(2)] the number of partitions of $n$ into parts congruent to $\pm 2$ modulo $5$ is the same as the number of partitions of $n$ such that every two consecutive parts have difference at least 2 and that the smallest part is greater than $1$.
\end{itemize}
\end{theorem}
To better understand the partition identities, let's introduce the necessary knowledge about integer partitions. For more details, readers can refer to Andrews' book \cite{andtp}.

For a given non-negative integer $n$, a {\it partition} $\lambda$ of $n$ is a weakly increasing list of positive integers that sum up to $n$. We write $\lambda=\lambda_{1}+\lambda_{2}+\cdots+\lambda_{m}=(\la_1, \la_2, ..., \la_m)$ with $\lambda_{1}\leq \lambda_{2}\cdots\leq\lambda_{m}$, where the {\it weight} of $\lambda$ will be denoted by $|\lambda|=n$ and each $\lambda_{i}$ is called a {\it part} of $\lambda$ for $1\leq i\leq m$. The number of parts $m$ is called the {\it length} of the partition $\lambda$ and is denoted by $\ell(\lambda)$. 
Denote by $\cP(n)$ the set of all partitions of $n$, while its cardinality $|\cP(n)|$ is denoted as $p(n)$ and we set $\cP:=\bigcup_{n\ge 0}\cP(n)$. There are several frequently studied subsets in the set $\cP$ as follows. Let $\cD(n)$ be the set of strict partitions of $n$ that are those with distinct parts, $\cRR(n)$ be the set of partitions of $n$ with consecutive parts differing by at least two, and $\cOD(n)$ be the set of partitions of $n$ with distinct odd parts. We set 
\begin{align*}
\cD:=\bigcup_{n\geq 0}\cD(n),\ \cRR:=\bigcup_{n\geq 0}\cRR(n),\text{ and }\cOD:=\bigcup_{n\geq 0}\cOD(n).
\end{align*}

For a given partition $\lambda$, its {\it Ferrers diagram}~\cite[p.~7]{andtp} is a graphical representation, denoted as $[\lambda]$, using left-justified rows of unit cells, such that the $i$-th row (from bottom up) consists of $i$ cells. For example, the Ferrers diagram of $2+3+3+5$ is shown in Figure~\ref{fig:Ferrers}. 
\begin{figure}[h!]
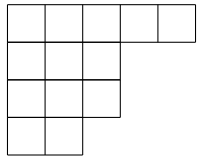

\begin{ferrers}
\addcellrows{5+3+3+2}
\end{ferrers}
\caption{The Ferrers diagram $[\lambda]$ for $\lambda=2+3+3+5$.}
\label{fig:Ferrers}
\end{figure}


So far, a large number of Rogers-Ramanujan type identities have been discovered, we study recent papers and notice the following parameterized Rogers-Ramanujan type identities that were found by Hao-Kuai-Xia \cite{HKX24} and Li-Wang \cite{LW24}.

\begin{theorem}[{cf. \cite[Thm.~2, Thm.~3 and Thm.~6]{HKX24}}]\label{thm:three HKX}
We have
\begin{align}
\sum_{i, j\geq 0}\frac{(-1)^j(x; q^2)_jy^{i+j}q^{i^2+j^2+2ij}}{(q^2; q^2)_i(q^2; q^2)_j}&=(yq; q^2)_{\infty}\sum_{n\geq 0}\frac{(-x; q^2)_{n}y^nq^{n^2}}{(yq, q^2; q^2)_n},\label{id:HKX1}\\
\sum_{i, j, k\geq 0}\frac{q^{\binom{k}{2}+\binom{i+j+k}{2}}(-1)^{i+j}x^{2i+2k+j}y^{2j+2k+i}}{(q; q)_i(q; q)_j(q; q)_k}&=(x^2y; q)_{\infty}\sum_{k\geq 0}\frac{(x; q)_k(-1)^kq^{\binom{k}{2}}(xy^2)^k}{(q; q)_k(x^2y; q)_k},\label{id:HKX2}\\
\sum_{k, j\geq 0}\frac{(-1)^kx^kq^{\binom{k+j}{2}+\binom{j}{2}+j}}{(q; q)_k(q; q)_j(q; q)_{k+j}}&=\frac{(x; q)_{\infty}}{(q; q)_{\infty}}.\label{id:HKX3}
\end{align}
\end{theorem}

\begin{theorem}[{cf. \cite[Eq.~(3.7)]{LW24}}]\label{thm:LW}
We have
\begin{align}
\sum_{i, j\geq 0}\frac{x^iq^{j^2+j+2ij+i}}{(q^2; q^2)_i(q^2; q^2)_j}=\frac{(-q^2; q^2)_{\infty}}{(xq; q^4)_{\infty}}.\label{id:LW}
\end{align}
\end{theorem}
It is worth noting that they use the integral method and constant term method to prove these identities. Then inspired by the partition versions of the two Rogers-Ramanujan identities \eqref{id:RR1} and \eqref{id:RR2}, a natural question is whether we can interpret both sides of the above identities \eqref{id:HKX1}-\eqref{id:LW} from the perspective of partitions and further provide a bijective proof. The answer is yes, and this is also one of the main purposes of this paper.

Moreover, Chen and Yin \cite{CY24} proved four new Rogers-Ramanujan type identities for double series by the constant term method and the identities \eqref{id:RR1}-\eqref{id:Rog2}.

\begin{theorem}[{cf. \cite[Thm.~1]{CY24}}]\label{thm:four CY}
We have
\begin{align}
\sum_{i, j\geq 0}\frac{(-1)^{\binom{i-j}{2}}q^{\frac{3i^2}{4}+\frac{ij}{2}+\frac{3j^2}{4}}}{(q; q)_i(q; q)_j}&=\frac{1}{(q^2, q^8; q^{10})_{\infty}},\label{id:CY1}\\
\sum_{i, j\geq 0}\frac{(-1)^{\binom{i-j}{2}}q^{\frac{3i^2}{4}+\frac{ij}{2}+\frac{3j^2}{4}+i+j}}{(q; q)_i(q; q)_j}&=\frac{1}{(q^4, q^6; q^{10})_{\infty}},\label{id:CY2}\\
\sum_{i, j\geq 0}\frac{(-1)^{j}q^{\frac{i^2}{4}+\frac{ij}{2}+\frac{j^2}{4}}}{(q^2; q^2)_i(q^2; q^2)_j}&=\frac{1}{(-q^2; q^2)_{\infty}(q, q^4; q^{5})_{\infty}},\label{id:CY3}\\
\sum_{i, j\geq 0}\frac{(-1)^{j}q^{\frac{i^2}{4}+\frac{ij}{2}+\frac{j^2}{4}+i+j}}{(q^2; q^2)_i(q^2; q^2)_j}&=\frac{1}{(-q^2; q^2)_{\infty}(q^2, q^3; q^{5})_{\infty}}.\label{id:CY4}
\end{align}
\end{theorem}





Based on the above Theorem \ref{thm:four CY}, we obtain the parameterized versions of the identities \eqref{id:CY1}-\eqref{id:CY4} as follows by the identities \eqref{id:RR1}-\eqref{id:Rog2} and constructing two bijections.
\begin{theorem}\label{thm:para_CY}
We have
\begin{align}
\sum_{i, j\geq 0}\frac{(-1)^jx^{i+j}q^{i^2+2ij+j^2}}{(q^8; q^8)_i(q^8; q^8)_j}&=\sum_{n\geq 0}\frac{x^{2n}q^{4n^2}}{(q^{16}; q^{16})_{n}},\label{id:para_CY1}\\
\sum_{i, j\geq 0}\frac{(-1)^{\binom{i-j}{2}}x^{i+j}q^{3i^2+2ij+3j^2}}{(q^4; q^4)_i(q^4; q^4)_j}&=\sum_{n\geq 0}\frac{x^{2n}q^{8n^2}}{(q^{8}; q^{8})_{n}}.\label{id:para_CY2}
\end{align}
\end{theorem}


The rest of this paper is organized as follows. In Section \ref{sec:comb pfHKX}, firstly we use ``base+increments'' framework to prove a theorem about strict odd partitions, and use it to provide a combinatorial proof of Theorem \ref{thm:three HKX}. In Section \ref{sec:comb pfLW}, we give an I-shape iterative bijection for completing the proof of Theorem \ref{thm:LW}. Moreover, we may use this idea to show another identity as an example. In Section \ref{sec:comb pfCY}, we construct two bijections for showing Theorem \ref{thm:para_CY} where a similar definition with ``sol'' on ordinary partitions will appear. Finally, in Section \ref{sec:conclusion} we briefly discuss the content of the next paper in this series, and several directions to be considered in the future.

\section{Combinatorial proofs of Theorem \ref{thm:three HKX}}\label{sec:comb pfHKX}
In this section, firstly we will provide the generating function for strict odd partitions (that is, partitions with distinct odd parts) with two partition statistics. We shall use the combinatorial framework described in the first paper of this series \cite{FL241} --- ``base+increments'' machinery to show this generating function, and this framework originates from the work of Kur\c{s}ung\"{o}z \cite{kur10}. Moreover, we will prove identity \eqref{id:HKX1} by the generating function mentioned above. Finally, we can provide combinatorial proofs for the identities \eqref{id:HKX2} and \eqref{id:HKX3}.

In the previous work \cite{FL241, FL242, FL243}, we defined the partition statistic ``$\sol$'', which counts the number of sequences of odd length in a strict partition $\la\in \cD$. Then we extended this definition to the Rogers-Ramanujan partitions in the set $\cRR$. Given $\la\in \cRR$, a maximal string of parts where adjacent parts differ by exactly two contained in $\la$ is a {\it $2$-sequence} of $\la$. We denote the number of $2$-sequences of odd length in $\la$ by $\sol_2(\la)$. In fact, notice that the set $\cOD$ is a subset of $\cRR$, so the statistic ``$\sol_2$'' can be smoothly defined in $\cOD$. For example, given a partition $\la=1+3+7+13+15+17\in \cOD$, then we can readily know that $\sol_2(\la)=2$ because the set of $2$-sequences is $\{1+3, 7, 13+15+17\}$. Further, we let $\cE_k$ denote the set of partitions into multiples of $k$ and set $\cE_{k,n}:=\{\la\in \cE_k: \ell(\la)\leq n\}$ for all positive integers $k\geq 1$. Note that if $k=1$, then we have $\cE_1=\cP$. 

For finding a double-sum generating function for $\cOD$ with the two partition statistics ``$\sol_2$'' and ``$\ell$ '' (i.e., the length of a partition), we firstly consider a triple $(\beta^{(i, j)}, \mu, \eta)$ where the base partition $\beta^{(i, j)}$ is 
\begin{align}\label{eq:base}
[1,3],[5,7],\cdots ,[4j-3, 4j-1], (4j+1), (4j+5), \cdots , (4j+4i-3)
\end{align} 
and $(\mu, \eta)\in \cE_{2, i}\times \cE_{4, j}$ for $i, j\geq 0$. For convenience, we refer to $[a+b]$ as a {\it pair} and $(c)$ as a {\it singleton} where $a, b, c$ are parts of a certain partition. Clearly note that
\begin{align*}
|\beta^{(i, j)}|=2i^2+4j^2+4ij-i,\ \ell(\beta^{(i, j)})=i+2j, \text{ and }\sol_2(\beta^{(i, j)})=i.
\end{align*}

On the other hand, by definition of this triple, we know that the increment pair $(\mu, \eta)$ will be generated by $1/(q^2; q^2)_i(q^4; q^4)_j$. Therefore, we have the double-sum generating function
\begin{align}
\sum_{\substack{i, j\geq 0\\(\beta^{(i, j)}, \mu, \eta)}}x^{\sol_2(\beta^{(i, j)})}y^{\ell(\beta^{(i, j)})}q^{|\beta^{(i, j)}|+|\mu|+|\eta|}=\sum_{i, j\geq 0}\frac{x^iy^{i+2j}q^{2i^2+4j^2+4ij-i}}{(q^2; q^2)_i(q^4; q^4)_j}
\end{align}
where the triple $(\beta^{(i, j)}, \mu, \eta)$ traverses the set $\{\beta^{(i, j)}\}\times \cE_{2, i}\times \cE_{4, j}$. Let $\cOD_{i, j}:=\{\la\in \cOD: \ell(\la)=i+2j, \text{ and } \sol_2(\la)=i\}$ for all $i, j\geq 0$. Based on the above discussion, we may construct a bijection as follows to provide the double-sum generating function for the set $\cOD_{i, j}$.

\begin{theorem}\label{thm:bij_ODgf}
For $i, j\geq 0$, there exists a bijection
\begin{align*}
\al=\al_{i, j}: \{\beta^{(i, j)}\}\times \cE_{2, i}\times \cE_{4, j}&\ri \cOD_{i, j}\\
(\beta^{(i, j)}, \mu, \eta)&\mapsto \la,
\end{align*}
such that $|\la|=|\beta^{(i, j)}|+|\mu|+|\eta|$, $\ell(\la)=\ell(\beta^{(i, j)})$, and $\sol_2(\la)=\sol_2(\beta^{(i, j)})$. Further, we have
\begin{align}
\sum_{\la\in \cOD}x^{\sol_2(\la)}y^{\ell(\la)}q^{|\la|}=\sum_{i, j\geq 0}\frac{x^iy^{i+2j}q^{2i^2+4j^2+4ij-i}}{(q^2; q^2)_i(q^4; q^4)_j}.\label{eq:ODgf}
\end{align}
\end{theorem}

\begin{proof}
It suffices to construct the bijection $\al$ for completing the proof of this theorem. The basic idea comes from the work of Fu and Li \cite[Theorem 1.4]{FL241}, so we may proceed in a similar way so that details will be omitted. For a given triple $(\beta^{(i, j)}, \mu, \eta)\in \{\beta^{(i, j)}\}\times \cE_{2, i}\times \cE_{4, j}$ where $\beta^{(i, j)}$ is defined in \eqref{eq:base} and $\mu=\mu_1+\mu_2+\cdots +\mu_i$ and $\eta=\eta_1+\eta_2+\cdots +\eta_j$ (if the number of parts in $\mu$ (resp. $\eta$) is smaller than $i$ (resp. $j$), then we can add zeros to fill the positions as parts), firstly we add all parts in $\mu$ to those corresponding singletons of $\beta^{(i, j)}$, that is,
\begin{align*}
\beta^{(i, j)}&=[1,3],[5,7],\cdots ,[4j-3, 4j-1], (4j+1), (4j+5), \cdots , (4j+4i-3)\\
&\xrightarrow{\text{add $\mu$}}[1,3],[5,7],\cdots ,[4j-3, 4j-1], (4j+1+\mu_1), (4j+5+\mu_2), \cdots , (4j+4i-3+\mu_i).
\end{align*}
Next, we shall move each pair $[4s-3, 4s-1]$ forward $\eta_s/2$ times for $1\leq s\leq j$. Notice that if the pair is moved too many times, then it will hit a singleton. Hence, we need to define the operations ``two forward moves'' (since $\eta_s/2$ is even for $1\leq s\leq j$ and the requirement is to maintain all parts as always being odd) and ``an adjustment'' in the pairs as follows:
\begin{align*}
&(\text{parts}\leq 4s-5), [\mathbf{4s-3}, \mathbf{4s-1}], (4s+1), (\text{parts}\geq 4s+5)\\
\xrightarrow{\text{two forward moves}}&(\text{parts}\leq 4s-5), [\mathbf{4s-1}, \mathbf{4s+1}], (4s+1), (\text{parts}\geq 4s+5)\\
\xrightarrow{\text{an adjustment}}&(\text{parts}\leq 4s-5), (4s-3), [\mathbf{4s+1}, \mathbf{4s+3}], (\text{parts}\geq 4s+5).
\end{align*}
From the aforementioned two steps, we obtain the map $\al$ and $\al((\beta^{(i, j)}, \mu, \eta))\in \cOD_{i, j}$. Furthermore, we can see that this map is reversible step by step, and so it is a bijection.
\end{proof}

Our next goal is to use above generating function \eqref{eq:ODgf} to prove the identities \eqref{id:HKX1}. The main idea is to construct a bijection (in fact, an involution) that implements ``positive and negative cancellation'', so that the set of fixed points is precisely the partition set $\cOD$ generated by the generating function \eqref{eq:ODgf}. In \cite{FL241}, we provide a bijective proof for the following identity discovered by Wei-Yu-Ruan \cite{WYR23}.
\begin{align}
\sum\limits_{i, j\geq 0}\frac{q^{i^2+2ij+2j^2-i-j}}{(q; q)_{i}(q^2; q^2)_{j}}x^{i}y^{2j}=(y; q)_{\infty}\sum\limits_{n\geq 0}\frac{(-x/y; q)_{n}}{(q; q)_{n}(y; q)_{n}}q^{\binom{n}{2}}y^{n}.\label{id:WYR_2para}
\end{align}
Then we observe that the right hand side of \eqref{id:HKX1} will appear by letting $x\ri xy$, $q\ri q^2$ and $y\ri yq$ in \eqref{id:WYR_2para}. We have
\begin{align}
\sum_{i, j\geq 0}\frac{x^iy^{i+2j}q^{2i^2+4j^2+4ij-i}}{(q^2; q^2)_i(q^4; q^4)_j}=(yq; q^2)_{\infty}\sum_{j\geq 0}\frac{(-x; q^2)_{j}y^jq^{j^2}}{(yq, q^2; q^2)_j}.\label{id:main_1}
\end{align}
Actually our bijection for proving \eqref{id:WYR_2para} can be used to prove this equation \eqref{id:main_1} with only minor modifications, as it only involves simple parameter variations. Therefore, to show the identities \eqref{id:HKX1} bijectively, it suffices to show the following theorem:
\begin{theorem}\label{thm:main_2}
We have
\begin{align}
\sum_{i, j\geq 0}\frac{x^iy^{i+2j}q^{2i^2+4j^2+4ij-i}}{(q^2; q^2)_i(q^4; q^4)_j}=\sum_{i, j\geq 0}\frac{(-1)^j(x; q^2)_jy^{i+j}q^{i^2+j^2+2ij}}{(q^2; q^2)_i(q^2; q^2)_j}.\label{id:main_2}
\end{align}
\end{theorem}

\begin{remark}
Let $x\ri -x$ and $y\ri -y$ in the right hand side of \eqref{id:main_1}, we have
\begin{align*}
(-yq; q^2)_{\infty}\sum_{j\geq 0}\frac{(x; q^2)_{j}(-y)^jq^{j^2}}{(-yq, q^2; q^2)_j}&=\sum_{j\geq 0}\frac{(x; q^2)_{j}(-y)^jq^{j^2}}{(q^2; q^2)_j}\cdot (-yq^{2j+1}; q^2)_{\infty}\\
&=\sum_{j\geq 0}\frac{(x; q^2)_{j}(-y)^jq^{j^2}}{(q^2; q^2)_j}\cdot\sum_{i\geq 0} \frac{y^iq^{i^2+2ij}}{(q^2; q^2)_i}=\text{RHS of \eqref{id:main_2}}
\end{align*} 
where the equation (II.2) of \cite{GR90} is used in the second step, that is, the $q$-exponential function: 
\begin{align}
\sum_{n\geq 0}\frac{q^{\binom{n}{2}}z^n}{(q; q)_n}=(-z; q)_{\infty}.\label{id:2qe}
\end{align}
However, after making this change $(x, y)\ri (-x, -y)$ in the identity \eqref{id:main_1}, the left hand side of it remains the original form, which also contributes to the proof of the identity \eqref{id:main_2}. In the involution we shall construct below, readers will also see why this change in the parameter symbol does not affect the generation of fixed points.
\end{remark}

Before presenting this involution, let's first review a concept that plays a crucial role in the construction process. Given a partition $\la=\la_1+\la_2+\cdots +\la_m\in \cP$, the $k$-th {\it tL-shape} for $1\leq k\leq m$ and $1\leq t\leq \la_k$ refers to the shaded portion of its Ferrers diagram $[\la]$ as shown below.
\begin{figure}[h!]
\begin{ferrers}
	    \addsketchrows{12+11+9+7+4+3+1}
	\addtext{0.9}{0.3}{$t$ columns}
          \addline{0.5}{0}{0.5}{-1.5}
          \addline{0}{-1.5}{3.5}{-1.5}
          \addline{0}{-1}{4.5}{-1}
          \highlightcellbycolor{1}{1}{black}
        \highlightcellbycolor{1}{2}{black}
        \highlightcellbycolor{1}{3}{black}
        \highlightcellbycolor{1}{4}{black}
        \highlightcellbycolor{2}{1}{black}
        \highlightcellbycolor{2}{2}{black}
        \highlightcellbycolor{2}{3}{black}
        \highlightcellbycolor{2}{4}{black}
        \highlightcellbycolor{3}{1}{black}
        \highlightcellbycolor{3}{2}{black}
           \highlightcellbycolor{3}{3}{black}
        \highlightcellbycolor{3}{4}{black}
        \highlightcellbycolor{3}{5}{black}
        \highlightcellbycolor{3}{6}{black}
                \highlightcellbycolor{3}{7}{black}
                        \highlightcellbycolor{3}{8}{black}
        \highlightcellbycolor{3}{9}{black}
        \addtext{-0.75}{-0.25}{$\lambda_{m}\rightarrow$}
        \addtext{-1}{-0.75}{$\vdots$}
\addtext{-0.75}{-1.25}{$\lambda_{k}\rightarrow$}
 \addtext{-1}{-2.25}{$\vdots$}
\addtext{-0.75}{-3.25}{$\lambda_{1}\rightarrow$}
\end{ferrers}
\caption{The $k$-th $t$L-shape in Ferrers diagram $[\la]$}
\label{fig:2L_shape}
\end{figure}

We denote the size of the $k$-th $t$L-shape as $s_{k, t}=s_{k, t}(\la)=\la_k+t(m-k)$. Now let's turn our attention back to the identity \eqref{id:main_2}. Firstly we need to translate the right hand side of this identity into the form of partition sets. Note that
\begin{align*}
\text{RHS of \eqref{id:main_2}}&=\sum_{j\geq 0}\frac{(x; q^2)_{j}(-y)^jq^{j^2}}{(q^2; q^2)_j}\cdot\sum_{i\geq 0} \frac{y^iq^{i^2+2ij}}{(q^2; q^2)_i}\\
&=\sum_{j\geq 0}\frac{(xy - y)(xy q^2-y)\cdots (xy q^{2j-2}-y)q^{1+3+\cdots +(2j-1)}}{(q^2; q^2)_j}\cdot (-yq^{2j+1}; q^2)_{\infty}.
\end{align*}
 According to this decomposition, we can interpret it as the generating function of the following set of weighted partition pairs $(\la, \mu)\in \cA^{\I}_{j}\times \cB^{\I}_{j}$. 
\begin{itemize}
\item[$\mathbf{(AI)}$] Let $\cA^{\I}_j$ be the set of partitions $\la$ with $j$ distinct odd parts, such that the parts are labeled as either $(xy)$ or $y$, and a part $\la_i$ can be labeled as $(xy)$ only when $\la_{i+1}-\la_i\geq 4$. Further, each partition in this set is assigned a sign $(-1)^m$ where $m$ is the number of parts labeled as $y$. We make the convention $\la_{j+1}=+\infty$, so that the largest part $\la_j$ can be labeled as either $(xy)$ or $y$.

\item[$\mathbf{(BI)}$] Let $\cB^{\I}_j$ be the set of partitions $\mu$ with distinct odd parts and each part not less than $(2j+1)$, and each part is labeled as $y$.
\end{itemize}

For example, we may take $(1+3_x+7+9+11_x, 11+13)\in \cA_{5}^{\I}\times \cB^{\I}_{5}$ where the subscript $x$ indicates this part is labeled as $xy$ and all remaining parts are labeled as $y$. Then this partition pair has weight $x^2y^7q^{55}$ with sign $(-1)^3=-1$. Note that the current weight includes the label. In the latter part of this paper, the weight has the same meaning as it does now. Denote the weight of $(\la, \mu)$ as $\wt(\la, \mu)$. Next we need following concepts to help us better construct this involution.
\begin{itemize}
\item[(1).] For $k\geq 1$, if there is a $2$-sequence with length $2k$ in a certain $\la\in \bigcup_{j\geq 0}\cA^{\I}_j$ where the first $(2k-1)$ parts are labeled as $y$ and the last one is labeled as $(xy)$, or a $2$-sequence with length $(2k-1)$ where all parts are labeled as $y$, then we call such a $2$-sequence {\it illegal}. Otherwise, call it {\it legal}. Hence, in the example above, the $2$-sequence $1+3_x$ is illegal and $7+9+11_x$ is legal.

\item[(2).] We can define a partition statistic $\sfi(\la)=s$ where $s$ is the size of $2$L-shape at the smallest part which is in the first illegal $2$-sequence of $\la\in\bigcup_{j\geq 0}\cA^{\I}_j$. Similarly, in the aforementioned example, we know that the smallest part in the first illegal $2$-sequence is $1$, then $\sfi(\la)=9$. Further, if there is no illegal $2$-sequence in $\la$, then we set $\sfi(\la)=+\infty$.
\end{itemize}
After introducing these useful notions, now we are going to construct the involution $\varphi$ on the set $\bigcup_{j\geq 0}\cA^{\I}_j\times \cB^{\I}_j$.

\begin{theorem}\label{thm:involution_1}
There exists an involution
\begin{align*}
\varphi: \bigcup_{j\geq 0}(\cA^{\I}_j\times \cB^{\I}_j)&\ri \bigcup_{j\geq 0}(\cA^{\I}_j\times \cB^{\I}_j)\\
(\la, \mu)&\mapsto (\hla, \hmu),
\end{align*}
such that $|\la|+|\mu|=|\hla|+|\hmu|$. And when $(\la, \mu)\neq (\hla, \hmu)$, they have the same weights and the opposite signs. Consequently, Theorem \ref{thm:main_2} holds true and \eqref{id:HKX1} follows.
\end{theorem}

\begin{proof}
For a given pair $(\la, \mu)\in \cA_j^{\I}\times \cB_j^{\I}$ where $\mu_1$ is the smallest part of $\mu$, if $\mu$ is the empty partition $\ep$, then we set $\mu_1=+\infty$. The core idea of this involution is to compare the size of $\sfi(\la)$ and $\mu_1$. However, if $\sfi(\la)=\mu_1=+\infty$, then we know that $\mu=\ep$ and there is no illegal $2$-sequence in $\la$ which implies that each $2$-sequence of odd length is labeled as $x$ and all parts are labeled as $y$. Therefore, in this case we set $(\la, \mu)=\varphi((\la, \mu))=(\hla, \hmu)$ which is the fixed point. Then we have
\begin{align*}
\sum_{\substack{(\la, \mu)\in \bigcup_{j\geq 0}(\cA^{\I}_j\times \cB^{\I}_j)\\\varphi((\la, \mu))=(\la, \mu)}}\wt(\la, \mu)=\sum_{\pi\in \cOD}x^{\sol_2(\pi)}y^{\ell(\pi)}q^{|\pi|}
\end{align*}
which implies that Theorem \ref{thm:main_2} holds true and identity \eqref{id:HKX1} follows. Hence, it suffices to show that when $(\la, \mu)\neq (\hla, \hmu)$ they have the same weights and the opposite signs by this involution $\varphi$. Assume that 
$$
\la=\la_1+\la_2+\cdots +\la_j\text{  and  }\mu=\mu_1+\mu_2+\cdots +\mu_{\ell(\mu)}.
$$
Now there are two cases as follows:
\begin{description}
\item[CASE I] If $\sfi(\la)<\mu_1$, then we delete the $2$L-shape at the smallest part (suppose $\la_a$ for a certain $a$) which is in the first illegal $2$-sequence of $\la$, and attach it as the new (smallest) part to $\mu$. That is, $\hla=\la_1+\cdots +\la_{a-1}+(\la_{a+1}-2)+\cdots +(\la_j-2)$ and $\hmu=\sfi(\la)+\mu_1+\cdots +\mu_{\ell(\mu)}$
where $\sfi(\la)=s_{a, 2}=\la_a+2j-2a$ and note that $\la_{a+1}-2\geq \la_a$. Hence we obtain the resulting pair $(\hla, \hmu)\in \cA^{\I}_{j-1}\times \cB^{\I}_{j-1}$. 

\item[CASE II] If $\sfi(\la)\geq \mu_1$, then we remove $\mu_1$ from $\mu$ and insert it into $\la$ as a $2$L-shape. In this case, we know that $\hmu=\mu_2+\cdots +\mu_{\ell(\mu)}$, and for $\hla$ there are two cases: (i) if $\la_1+2j\geq\mu_1$, then we have $\hla=(\mu_1-2j)+(\la_1+2)+\cdots +(\la_j+2)$; (ii) otherwise, there must exist the largest part $\la_b$ such that $\la_b+2j-2b+2<\mu_1$. Then $\hla=\la_1+\cdots +\la_b+(\mu_1-2j+2b)+(\la_{b+1}+2)+\cdots (\la_j+2)$. Hence we obtain the resulting pair $(\hla, \hmu)\in \cA^{\I}_{j+1}\times \cB^{\I}_{j+1}$.
\end{description}
Next to explain the well-definedness of this involution, we leave the verification of the following items to the reader.
\begin{itemize}
\item[(1).] Easily see that in cases I and II $\varphi$ is well-defined; i.e., $(\hla, \hmu)\in \bigcup_{j\geq 0}(\cA^{\I}_j\times \cB^{\I}_j)$ in both cases.

\item[(2).] If $(\la, \mu)$ is in case I, then its image $(\hla, \hmu)$ is in case II; if $(\la, \mu)$ is in case II, then its image $(\hla, \hmu)$ is in case I. Moreover, we have $\varphi^2((\la, \mu))=(\la, \mu)$ for all pair $(\la, \mu)\in \bigcup_{j\geq 0}(\cA^{\I}_j\times \cB^{\I}_j)$.

\item[(3).] For cases I and II, $(\la, \mu)$ and $(\hla, \hmu)$ have the same weight but opposite sign.
\end{itemize}



\end{proof}

\begin{example}
In the same way, denote the empty partition as $\ep$ and denote a part $\la_i$ labeled as $(xy)$ by $(\la_i)_x$. Now we consider all partition pairs $(\la, \mu)\in \bigcup_{j\geq 0}(\cA^{\I}_j\times \cB^{\I}_j)$ with $|\la|+|\mu|=18$. Firstly we can determine the set of fixed points under the involution $\varphi$:
\begin{align*}
\{(1_x+17_x, \ep),\ (3_x+15_x, \ep),\ (5_x+13_x, \ep),\ (7_x+11_x, \ep),\ (1+3+5_x+9_x, \ep)\},
\end{align*}
and then for the remaining partition pairs we have
\\

\begin{minipage}[c]{0.5\textwidth}
\centering
\begin{tabular}{r|l}
sign=$-1$ & sign=$+1$\\
\hline
$(1, 17)$ & $(\ep, 1+17)$\\
$(1, 3+5+9)$ & $(\ep, 1+3+5+9)$\\
$(3, 15)$ & $(\ep, 3+15)$\\
$(3, 3+5+7)$ & $(1+5, 5+7)$\\
$(5, 13)$ & $(\ep, 5+13)$\\
$(7, 11)$ & $(\ep, 7+11)$\\
$(9, 9)$ & $(7+11, \ep)$\\
$(11, 7)$ & $(5+13, \ep)$\\
$(13, 5)$ & $(3+15, \ep)$\\
$(15, 3)$ & $(1+17, \ep)$\\
$(1+3_x, 5+9)$ & $(1_x, 3+5+9)$\\
$(1+5_x, 5+7)$ & $(3_x, 3+5+7)$\\
$(1_x+5, 5+7)$ & $(1+3_x+7, 7)$\\
\end{tabular}
\end{minipage}
\begin{minipage}[c]{0.5\textwidth}
\centering
\begin{tabular}{r|l}
sign=$-1$ & sign=$+1$\\
\hline
$(1+3+5, 9)$ & $(1+3, 5+9)$\\
$(1+3+7, 7)$ & $(1+3+5+9, \ep)$\\
$(1+3_x+7_x, 7)$ & $(1_x+5_x, 5+7)$\\
$(1+17_x, \ep)$ & $(15_x, 3)$\\
$(1_x+17, \ep)$ & $(1_x, 17)$\\
$(3_x+15, \ep)$ & $(3_x, 15)$\\
$(3+15_x, \ep)$ & $(13_x, 5)$\\
$(5_x+13, \ep)$ & $(5_x, 13)$\\
$(5+13_x, \ep)$ & $(11_x, 7)$\\
$(7_x+11, \ep)$ & $(7_x, 11)$\\
$(7+11_x, \ep)$ & $(9_x, 9)$\\
$(1+3+5+9_x, \ep)$ & $(1+3+7_x, 7)$\\
$(1+3+5_x+9, \ep)$ & $(1+3+5_x, 9)$\\
\end{tabular}
\end{minipage}
\end{example}

In the rest of this section, we shall use the core techniques involved in the two bijections above to complete the proofs of the remaining two identities in Theorem \ref{thm:three HKX}, namely \eqref{id:HKX2} and $\eqref{id:HKX3}$.

\begin{proof}[Combinatorial proof of \eqref{id:HKX2}]
To describe the bijection more conveniently, we first make a parameter change. Let $x\rightarrow -x$ and $y\rightarrow -yq$ in \eqref{id:HKX2}, then we have
\begin{align}
\sum_{i, j, k\geq 0}\frac{q^{\binom{k}{2}+\binom{i+j+k}{2}+2j+2k+i}x^{2i+2k+j}y^{2j+2k+i}}{(q; q)_i(q; q)_j(q; q)_k}=(-x^2yq; q)_{\infty}\sum_{m\geq 0}\frac{(-x; q)_mq^{\binom{m}{2}+2m}(xy^2)^m}{(q; q)_m(-x^2yq; q)_m}.\label{id:HKX2-pc}
\end{align}
Subsequently, we will provide combinatorial interpretations for both sides of this identity \eqref{id:HKX2-pc}, respectively.

For the left hand side, we have
\begin{align*}
\text{LHS}&=\sum_{j, k\geq 0}\frac{q^{(j^2+3j)/2+k^2+k+jk}(xy^2)^j(x^2y^2)^k}{(q; q)_j(q; q)_k}\cdot\sum_{i\geq 0}\frac{q^{(i^2+i)/2+ij+ik}(x^2y)^i}{(q; q)_i}\\
&=\sum_{j, k\geq 0}\frac{q^{(j^2+3j)/2+k^2+k+jk}(xy^2)^j(x^2y^2)^k}{(q; q)_j(q; q)_k}\cdot (-x^2yq^{j+k+1}; q)_{\infty}.
\end{align*}
According to this decomposition, we can interpret it as a set of partition quadruples $(\beta^{(j, k)}, \mu, \eta, \ga)\in \{\beta^{(j, k)}\}\times \cP_j\times \cP_k\times \cB_{j, k}^{\II_{\L}}$ for $j, k\geq 0$.
\begin{itemize}
\item[$\mathbf{(AII_{\L})}$] Let $\cA_{j, k}^{\II_{\L}}=\{\beta^{(j, k)}\}\times \cP_j\times \cP_k$ be the set of triples $(\beta^{(j, k)}, \mu, \eta)$ where the base partition $\beta^{(j, k)}$ is
\begin{align*}
(2), (3), \cdots , (j+1), (j+2), (j+4),\cdots ,(j+2k)
\end{align*}
and $(\mu, \eta)\in \cP_j\times \cP_k$. Moreover, each of parts $(2), (3), ..., (j+1)$ is labeled as $(xy^2)$ and each of parts $(j+2), (j+4), ..., (j+2k)$ is labeled as $(x^2y^2)$.

\item[$\mathbf{(BII_{\L})}$] Let $\cB_{j, k}^{\II_{\L}}$ be the set of partitions $\ga$ of $n$ with distinct parts satisfying that the smallest part is not less than $(j+k+1)$ and each part is labeled as $(x^2y)$.
\end{itemize}


For the right hand side, we have
\begin{align*}
\text{RHS}=\sum_{k\geq 0}\frac{(xy^2+x^2y^2)(xy^2+x^2y^2q)\cdots (xy^2+x^2y^2q^{k-1})q^{\frac{k^2}{2}+\frac{3k}{2}}}{(q; q)_k}\cdot(-x^2yq^{k+1}; q)_{\infty}.
\end{align*}
According to this decomposition, we can interpret it as a set of partition pairs $(\la, \pi)\in \cA_{m}^{\II_{\R}}\times \cB_{m}^{\II_{\R}}$ for $m\geq 0$.
\begin{itemize}
\item[$(\mathbf{AII}_{\R})$] Let $\cA_m^{\II_\mathrm{R}}$ be the set of partitions $\la$ with $m$ distinct parts satisfying that the smallest part is $\geq 2$, if $\la_i$ is labeled as $(x^2y^2)$ then $\la_{i+1}-\la_i\geq 2$ and other parts are labeled as $(xy^2)$. Note that let $\la_{m+1}=+\infty$ make sure that $\la_m$ can be labeled as $(x^2y^2)$ or $(xy^2)$.

\item[$(\mathbf{BII}_{\R})$] Let $\cB_m^{\II_{\R}}$ be the set of partitions $\pi$ with distinct parts satisfying that the smallest part is not less than $m+1$ and each part is labeled as $(x^2y)$.
\end{itemize}
By the combinatorial interpretations above, we see that for any $\la\in \cA_m^{\II_\mathrm{R}}$, if the number of parts labeled as $(xy^2)$ is $j$ and the number of parts labeled as $(x^2y^2)$ is $k$ then we may define $\cA_{j, k}^{\II_\mathrm{R}}=\{\la\in \cA_m^{\II_\mathrm{R}}: m=j+k, j \text{ parts labeled as }(xy^2)\text{ and }k\text{ parts labeled as }(x^2y^2)\}$. Hence, $\cB_{j, k}^{\II_{\R}}:=\cB_m^{\II_{\R}}$ with $m=j+k$.

Based on the above discussion, it suffices to construct the bijection between $\cA_{j, k}^{\II_\mathrm{L}}$ and $\cA_{j, k}^{\II_\mathrm{R}}$, which implies the bijection between $\cA_{j, k}^{\II_\mathrm{L}}\times \cB_{j, k}^{\II_\mathrm{L}}$ and $\cA_{j, k}^{\II_\mathrm{R}}\times \cB_{j, k}^{\II_\mathrm{R}}$ if we fix the $\pi=\ga$. Finally we build the bijection
\begin{align*}
\tau: \cA_{j, k}^{\II_{\L}}=\{\beta^{(j, k)}\}\times \cP_j\times \cP_k&\ri\cA_{j, k}^{\II_{\R}}\\
(\beta^{(j, k)}, \mu, \eta)&\mapsto \la
\end{align*}
such that $|\beta^{(j, k)}|+|\mu|+|\eta|=|\la|$ where $\mu=\mu_1+\mu_2+\cdots +\mu_j$ and $\eta=\eta_1+\eta_2+\cdots +\eta_k$. Firstly we add all parts in $\eta$ to those corresponding parts of $\beta^{(j, k)}$, that is,
\begin{align*}
\beta^{(j, k)}\xri{\text{add }\eta}(2)+(3)+\cdots +(j+1)+(j+2+\eta_1)+(j+4+\eta_2)+\cdots +(j+2k+\eta_k).
\end{align*}
Next we will move each parts $(s+1)$ forward $\mu_s$ times for $1\leq s\leq j$ by the following rules:
\begin{equation}
\begin{array}{rll}
&(\text{parts}\leq s), \mathbf{(s+1)}, (s+2), (\text{parts}\geq s+4)\\
\xri{\text{one forward move}}&(\text{parts}\leq s), \mathbf{(s+2)}, (s+2), (\text{parts}\geq s+4)\\
\xri{\text{an adjustment}}&(\text{parts}\leq s), (s+1), \mathbf{(s+3)}, (\text{parts}\geq s+4).
\end{array} \label{rule:singleton}
\end{equation}
We get $\la$ with $|\la|=|\beta^{(j, k)}|+|\mu|+|\eta|$ by a series of operations above. Easily see that $\la\in \cA_{j, k}^{\II_{\R}}$ since this map does not change the label of any part and $\la$ satisfies all conditions in $\cA_{j, k}^{\II_{\R}}$. Moreover, $\tau$ is reversible step by step which implies it is a bijection.
\end{proof}

Finally, we briefly provide the combinatorial proof of \eqref{id:HKX3}. Similarly, we let $x\rightarrow -xq$ and multiply $(q; q)_{\infty}$ on both sides, then we have
\begin{align}
\sum_{i, j\geq 0}\frac{x^iq^{\frac{i^2+i}{2}+ij+j^2}}{(q; q)_i(q; q)_j}\cdot(q^{i+j+1}; q)_{\infty}=(-xq; q)_{\infty}=\sum_{\la\in \cD}x^{\ell(\la)}q^{|\la|}.\label{id:HKX3_pc}
\end{align}
Now we focus on the left hand side of this identity \eqref{id:HKX3_pc}. Firstly the exponent of $q$ in numerator means the base partition
\begin{align*}
\beta^{(i, j)}=(1), (2), \cdots , (i), (i+1), (i+3), \cdots ,(i+2j-1)
\end{align*}
 and each part of $(1), (2), ..., (i)$ is labeled as $x$. Moreover, by the same operations with the bijection $\tau$, we may obtain a bijection between $\{\beta^{(i, j)}\}\times \cP_i\times \cP_j$ and $\cA_{i, j}^{\III}$ defined as follows:
 \begin{itemize}
\item[$\mathbf{(AIII)}$] Let $\cA_{i, j}^{\III}$ be the set of strict partitions $\la$ with $i$ parts labeled as $x$ and $j$ parts not labeled, satisfying that if $\la_k$ is not labeled then $\la_{k+1}-\la_k\geq 2$ and other parts are labeled as $x$. Note that let $\la_{i+j+1}=+\infty$ make sure that $\la_{i+j}$ can be labeled as $x$ or has no label.
\end{itemize}
Further, the term $(q^{i+j+1}; q)_{\infty}$ will generate the following partition set:
\begin{itemize}
\item[$\mathbf{(BIII)}$] Let $\cB_{i, j}^{\III}$ be the set of strict partitions $\mu$ satisfying that the smallest part is not less than $(i+j+1)$ and each partition is assigned with a sign $(-1)^{\ell(\mu)}$.
\end{itemize}
Lastly, we can use the L-shape (i.e., $t$L-shape when $t=1$) to construct an involution on $\bigcup_{i, j\geq 0}\cA_{i, j}^{\III}\times \cB_{i, j}^{\III}$ that is extremely similar with $\varphi$ appearing in Theorem \ref{thm:involution_1}. This will mean that the remaining set of fixed points will be $\cD$ with each part labeled as $x$. We will not elaborate on the details here, and interested readers can try to fill in the gaps. At the end of this section, we provide an example to illustrate this involution and the set of fixed points.


\begin{example}
We consider all partition pairs $(\la, \mu)\in \bigcup_{i, j\geq 0}\cA_{i, j}^{\III}\times \cB_{i, j}^{\III}$ with $|\la|+|\mu|=9$. Firstly there are $8$ fixed points as follows:
\begin{align*}
&(9_x, \ep), (1_x+8_x, \ep), (2_x+7_x, \ep), (3_x+6_x, \ep), (4_x+5_x, \ep), (1_x+2_x+6_x, \ep),\\
&(1_x+3_x+5_x, \ep), (2_x+3_x+4_x, \ep).
\end{align*}
They are all strict partitions of $9$ if we only read $\la$. And then we present a one-to-one correspondence of positive and negative offsets for the remaining partitions.
\\

\begin{minipage}[c]{0.5\textwidth}
\centering
\begin{tabular}{r|l}
sign=$-1$ & sign=$+1$\\
\hline
$(\ep, 9)$ & $(9, \ep)$\\
$(\ep, 1+2+6)$ & $(1, 2+6)$\\
$(\ep, 1+3+5)$ & $(1, 3+5)$\\
$(\ep, 2+3+4)$ & $(2, 3+4)$\\
$(7, 2)$ & $(1+8, \ep)$\\
$(7_x, 2)$ & $(1+8_x, \ep)$\\
$(6, 3)$ & $(2+7, \ep)$\\
$(6_x, 3)$ & $(2+7_x, \ep)$\\
$(5, 4)$ & $(3+6, \ep)$\\
$(5_x, 4)$ & $(3+6_x, \ep)$\\
$(4, 5)$ & $(\ep, 4+5)$\\
$(4_x, 5)$ & $(4_x+5, \ep)$\\
$(3, 6)$ & $(\ep, 3+6)$\\
$(3_x, 6)$ & $(3_x+6, \ep)$\\
$(2, 7)$ & $(\ep, 2+7)$\\
$(2_x, 7)$ & $(2_x+7, \ep)$\\
$(1, 8)$ & $(\ep, 1+8)$\\
$(1_x, 8)$ & $(1_x+8, \ep)$\\
$(1+5, 3)$ & $(4, 2+3)$
\end{tabular}
\end{minipage}
\begin{minipage}[c]{0.5\textwidth}
\centering
\begin{tabular}{r|l}
sign=$-1$ & sign=$+1$\\
\hline
$(1_x+5, 3)$ & $(1_x+2+6, \ep)$\\
$(1+5_x, 3)$ & $(4_x, 2+3)$\\
$(1_x+5_x, 3)$ & $(1_x+2+6_x, \ep)$\\
$(2+4, 3)$ & $(1+3+5, \ep)$\\
$(2_x+4, 3)$ & $(1+3_x+5, \ep)$\\
$(2+4_x, 3)$ & $(1+3+5_x, \ep)$\\
$(2_x+4_x, 3)$ & $(1+3_x+5_x, \ep)$\\
$(1+4, 4)$ & $(3, 2+4)$\\
$(1_x+4, 4)$ & $(1_x+3+5, \ep)$\\
$(1+4_x, 4)$ & $(3_x, 2+4)$\\
$(1_x+4_x, 4)$ & $(1_x+3+5_x, \ep)$\\
$(2_x+3, 4)$ & $(2_x, 3+4)$\\
$(2_x+3_x, 4)$ & $(2_x+3_x+4, \ep)$\\
$(1+3, 5)$ & $(2, 2+5)$\\
$(1_x+3, 5)$ & $(1_x, 3+5)$\\
$(1+3_x, 5)$ & $(2_x, 2+5)$\\
$(1_x+3_x, 5)$ & $(1_x+3_x+5, \ep)$\\
$(1_x+2, 6)$ & $(1_x, 2+6)$\\
$(1_x+2_x, 6)$ & $(1_x+2_x+6, \ep)$
\end{tabular}
\end{minipage}

\end{example}

\section{An I-shape iterative bijection for Theorem \ref{thm:LW} and its application}\label{sec:comb pfLW}

In this section, we will introduce an I-shape iterative bijection to prove Theorem \ref{thm:LW} and further use this idea to complete a combinatorial proof of another identity as an application. For a given partition $\la=\la_1+\la_2+\cdots +\la_m\in \cP$, the $k$-th {\it I-shape} for $1\leq k\leq m$ refers to the shaded portion of its Ferrers diagram $[\la]$ as shown in Figure \ref{fig:I_shape}.
 \begin{figure}[h!]
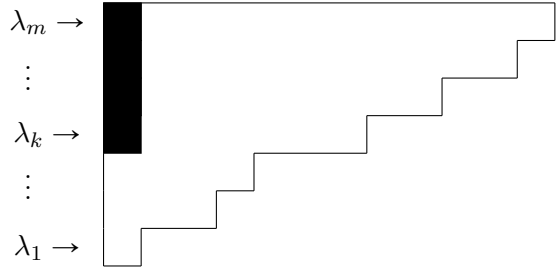

\begin{ferrers}
	    \addsketchrows{12+11+9+7+4+3+1}
        \addline{0.5}{0}{0.5}{-1.5}
      
        \highlightcellbycolor{1}{1}{black}
        \highlightcellbycolor{2}{1}{black}
        \highlightcellbycolor{3}{1}{black}
        \highlightcellbycolor{4}{1}{black}
    
        \addtext{-0.75}{-0.25}{$\lambda_{m}\rightarrow$}
        \addtext{-1}{-0.9}{$\vdots$}
\addtext{-0.75}{-1.75}{$\lambda_{k}\rightarrow$}
 \addtext{-1}{-2.35}{$\vdots$}
\addtext{-0.75}{-3.25}{$\lambda_{1}\rightarrow$}
\end{ferrers}
\caption{The $k$-th I-shape in Ferrers diagram $[\la]$}
\label{fig:I_shape}
\end{figure}

We denote the size of the $k$-th I-shape of $\la$ as $g_k=g_{k}(\la)=\ell(\la)-k+1=m-k+1$. Now we focus on constructing the bijection for identity \eqref{id:LW}. Firstly, for the right side of \eqref{id:LW}, we let $\cED$ be the set of partitions with distinct even parts, and let $\cF$ be the set of (weighted) partitions with each part $\equiv 1\ (\mathrm{mod}\ 4)$ and labeled as $x$. Then we have
\begin{align*}
\sum_{(\la, \mu)\in \cF\times \cED}x^{\ell(\la)}q^{|\la|+|\mu|}=\frac{(-q^2; q^2)_{\infty}}{(xq; q^4)_{\infty}}.
\end{align*}
On the other hand, for the left hand side of \eqref{id:LW}, note that 
\begin{align*}
\sum_{i, j\geq 0}\frac{x^iq^{j^2+j+2ij+i}}{(q^2; q^2)_i(q^2; q^2)_j}=\sum_{i\geq 0}\frac{(xq)^{1+1+\cdots+1}}{(q^2; q^2)_i}\cdot \sum_{j\geq 0}\frac{q^{(2i+2)+(2i+4)+\cdots +(2i+2j)}}{(q^2; q^2)_{j}}.
\end{align*}
Then according to this decomposition and the ``base+increments'' framework, we can interpret it as the following set of (weighted) partitions $\cA_i^{\IV}\times \cB_i^{\IV}$.
\begin{itemize}
\item[$\mathbf{(AIV)}$] Let $\cA_i^{\IV}$ be the set of partitions $\beta$ with $i$ odd parts and each part is labeled as $x$.

\item[$\mathbf{(BIV)}$] Let $\cB_{i}^{\IV}$ be the set of partitions $\ga$ with distinct even parts such that the smallest part is not less than $(2i+2)$.
\end{itemize}

Next we will construct the bijection between $\bigcup_{i\geq 0}(\cA_i^{\IV}\times \cB_i^{\IV})$ and $\cF\times \cED$ by the notation of I-shape. Actually, it can be regarded as an iterative algorithm.

\begin{theorem}\label{thm:bij_LW}
There exists a bijection
\begin{align*}
\theta_1: \bigcup_{i\geq 0}(\cA_i^{\IV}\times \cB_i^{\IV})&\ri \cF\times \cED\\
(\la, \mu)& \mapsto (\beta, \gamma),
\end{align*}
such that $|\la|+|\mu|=|\beta|+|\gamma|$ and $\ell(\la)=\ell(\beta)$. Consequently, Theorem \ref{thm:LW} holds true.
\end{theorem}

\begin{proof}
Firstly we construct the map $\theta_1$. For given $(\la, \mu)\in\cA_i^{\IV}\times \cB_i^{\IV}$ where $\la=\la_1+ \la_2+\cdots+ \la_i$ and $\mu=\mu_1+\mu_2+\cdots + \mu_{\ell(\mu)}$,  we can divide it into two cases as follows.
\begin{description}
\item[CASE I] If each part $\la_j$ (for $1\leq j\leq i$) is congruent to $1$ modulo $4$, then we set $(\beta, \ga)=\theta_1((\la, \mu))=(\la, \mu)\in \cF\times \cED$.


\item[CASE II] If there are parts of $\la$ being congruent to $3$ modulo $4$, then we will present an iterative algorithm as follows.  
\begin{description}
\item[Initial setup] We set $\la=\la^{(0)}=\la^{(0)}_1+\cdots +\la^{(0)}_i$ where $\la^{(0)}_j=\la_j$ for all $1\leq j\leq i$, and further we set $\mu=\mu^{(0)}$.

\item[Step (s)] (for $1\leq s\leq i$) Search for the first part of $\la^{(s-1)}$ which is congruent to $3$ modulo $4$ from left to right and denote it as $\la_{t_{s}}^{(s-1)}$, then we remove the I-shape at $\la_{t_s}^{(s-1)}$ two times to obtain $2g_{t_s}=2i-2t_s+2$ and add $(2g_{t_s})$ as a new part into $\mu^{(s-1)}$. That is, 
\begin{align*}
\la^{(s-1)}&\mapsto \la^{(s)}=\la^{(s)}_1+\cdots +\la_i^{(s)},\\
\mu^{(s-1)}&\mapsto \mu^{(s)}=(2g_{t_s})+\text{(the parts of $\mu^{(s-1)}$)}
\end{align*}
where for each $j$ we set 
\begin{align*}
\la_j^{(s)}=\left\{\begin{array}{ll}
\la_j^{(s-1)} & \quad\text{ if }1\leq j<t_s,\\
\la_j^{(s-1)}-2 & \quad\text{ if }t_s\leq j\leq i.
\end{array}\right.
\end{align*}
Note that these operations are well-defined since $2g_{t_s}<\mu_1^{(s-1)}$ and $\la_{t_s}^{(s-1)}-\la_{t_{s}-1}^{(s-1)}\geq 2$ and $\la_j^{(s)}$ is still congruent to $1$ or $3$ modulo $4$ for each $j$. 


\item[Final setup] Assuming that the above program is terminated after $p\  (\leq i)$ steps, we obtain the partition pair $(\la^{(p)}, \mu^{(p)})$ where there is no part in $\la^{(p)}$ being congruent $3$ modulo $4$.
\end{description}
Finally, we obtain that $\theta_1((\la, \mu))=(\beta, \gamma)=(\la^{(p)}, \mu^{(p)})\in \cF\times \cED$. And further we have $|\la|+|\mu|=|\beta|+|\ga|$ and $\ell(\la)=\ell(\beta)$ which implies that they have the same weight.
\end{description}

On the other hand, we can construct the inverse map $\theta_1^{-1}$ step by step. For a given $(\beta, \gamma)\in \cF\times \cED$ where $\beta=(\beta_1, \beta_2, ..., \beta_{\ell(\beta)})$ and $\gamma=(\gamma_1, \gamma_2, ..., \gamma_{\ell(\ga)})$. We can divide it into two cases as follows.
\begin{description}
\item[CASE I'] If the part $\ga_1$ is not less than $(2\ell(\beta)+2)$, then we set $(\la, \mu)=\theta^{-1}((\beta, \ga))=(\beta, \ga)\in\cA_{\ell(\beta)}^{\IV}\times \cB_{\ell(\beta)}^{\IV}$.

\item[CASE II'] If there are parts being less than $(2\ell(\beta)+2)$ in $\gamma$, and suppose that $\{\gamma_1, ..., \gamma_t\}\ (t\leq \ell(\ga))$ is the set of all those parts, then we get $\mu=(\gamma_{t+1}, ..., \gamma_{\ell(\ga)})$. Moreover, we add the I-shape of size $\ga_s/2$ two times into $\beta$ where $s$ is from $t+1$ to $\ell(\ga)$ and obtain $\la$. Therefore, it is obvious that $(\la, \mu)\in \cA_{\ell(\beta)}^{\IV}\times \cB_{\ell(\beta)}^{\IV}$, this means that the inverse map $\theta^{-1}_1$ is well-defined.

\end{description}
Based on the above discussion, we easily that the map $\theta$ is a weight-preserving bijection. This is what we want to show.



\end{proof}

Next we consider an example to understand this bijection. We shall use boldface to indicate the parts we operate. Given $(\la, \mu)=(3+3+5+9+9+15, 14+16)\in \cA_{6}^{\IV}\times \cB_{6}^{\IV}$ with $|\la|+|\mu|=74$ and $\ell(\la)=6$, then we have
\begin{align*}
({\bf 3}+3+5+9+9+15, 14+16)&\xri{\text{Step (1)}}(1+1+{\bf 3}+7+7+13, 12+14+16)\\
&\xri{\text{Step (2)}}(1+1+1+5+5+{\bf 11}, 8+12+14+16)\\
&\xri{\text{Step (2)}}(1+1+1+5+5+9, 2+8+12+14+16).
\end{align*}
Hence, we obtain $(\beta, \ga)=\theta_1((\la, \mu))=(1+1+1+5+5+9, 2+8+12+14+16)\in \cF\times \cED$.

\begin{remark}
There exists an analytic proof for \eqref{id:LW}. In fact, we have
\begin{align*}
\sum_{i, j\geq 0}\frac{x^iq^{j^2+j+2ij+i}}{(q^2; q^2)_i(q^2; q^2)_j}&=\sum_{i\geq 0}\frac{x^iq^i}{(q^2; q^2)_{i}}\cdot \sum_{j\geq 0}\frac{q^{j^2+j+2ij}}{(q^2; q^2)_j}=\sum_{i\geq 0}\frac{x^iq^i}{(q^2; q^2)_{i}}\cdot(-q^{2i+2}; q^2)_{\infty}\\
&=\sum_{i\geq 0}\frac{x^iq^i}{(q^4; q^4)_{i}}\cdot(-q^2; q^2)_{\infty}=\frac{(-q^2; q^2)_{\infty}}{(xq; q^4)_{\infty}},
\end{align*}
where we use \eqref{id:2qe} and another $q$-exponential function (II.1) in \cite{GR90}:
\begin{align}
\sum_{n\geq 0}\frac{z^n}{(q; q)_n}=\frac{1}{(z; q)_{\infty}}.\label{id:1qe}
\end{align}
Moreover, with the help of this analytic proof and the bijection described above, we can obtain two generalizations of \eqref{id:LW}: 
\begin{align}
\sum_{i, j\geq 0}\frac{x^i y^j q^{j^2+j+2ij+i}}{(q^2; q^2)_{i}(q^2; q^2)_j}&=(-yq^2; q^2)_{\infty}\sum_{i\geq 0}\frac{x^i q^i}{(-yq^2, q^2; q^2)_i},\\
\sum_{i, j\geq 0}\frac{x^iq^{kj^2/2+kj/2+kij+i}}{(q^k; q^k)_{i}(q^k; q^k)_j}&=\frac{(-q^k; q^k)_{\infty}}{(xq; q^{2k})_{\infty}}\quad \text{ (for any $k\in \bN$)}.
\end{align}
\end{remark}

We will conclude this section with an application as follows. Now we consider the pair of partition set $\cA_i^{\V}\times \cB_i^{\V}$ for $i\geq 0$ defined as:
\begin{itemize}
\item[$\mathbf{(AV)}$] Let $\cA_i^{\V}$ be the set of partitions $\la$ with $(2i)$ or $(2i-1)$ parts satisfies that the difference between adjacent parts is at least $2$. And if the number of parts is $(2i-1)$, then the smallest part is more than $1$.

\item[$\mathbf{(BV)}$] Let $\cB_i^{\V}$ be the set of strict partitions $\mu$ satisfies that the smallest part is not less than $(2i+1)$. 
\end{itemize}
Then we can obtain the following theorem.



\begin{theorem}\label{thm:WZ}
There exists a bijection
\begin{align*}
\theta_2: \bigcup_{i\geq 0}(\cA_i^{\V}\times \cB_{i}^{\V})&\ri \cED\times \cD\\
(\la, \mu)& \mapsto (\beta, \gamma),
\end{align*}
such that $|\la|+|\mu|=|\beta|+|\gamma|$.
\end{theorem} 

\begin{proof}
The overall structure of this proof is similar to $\theta_1$, so we only give the construction of the iterative algorithm here. Furthermore, without loss of generality, we consider $(\la, \mu)\in \cA_i^{V}\times \cB_i^{\V}$ where $\la=\la_1+\cdots +\la_{2i}$ and $\mu=\mu_1+\cdots+\mu_{\ell(\mu)}$. Now there are two cases:
\begin{description}
\item[Case I] If there is no odd part in $\la$, then we set $(\beta, \ga)=\theta_2((\la, \mu))=(\la, \mu)\in \cED\times \cD$.

\item[Case II] If there are odd parts in $\la$, then we will present an iterative algorithm as follows.  
\begin{description}
\item[Initial setup] We set $\la=\la^{(0)}=\la^{(0)}_1+\cdots +\la^{(0)}_{2i}$ where $\la^{(0)}_j=\la_j$ for all $1\leq j\leq 2i$, and further we set $\mu=\mu^{(0)}$.

\item[Step (s)] (for $1\leq s\leq 2i$) Search for the first odd part of $\la^{(s-1)}$ from left to right and denote it as $\la_{t_{s}}^{(s-1)}$, then we remove the I-shape at $\la_{t_s}^{(s-1)}$ to obtain $g_{t_s}=i-t_s+1$ and add $g_{t_s}$ as a new part into $\mu^{(s-1)}$. That is, 
\begin{align*}
\la^{(s-1)}&\mapsto \la^{(s)}=\la^{(s)}_1+\cdots +\la_{2i}^{(s)},\\
\mu^{(s-1)}&\mapsto \mu^{(s)}=g_{t_s}+\text{(the parts of $\mu^{(s-1)}$)}
\end{align*}
where for each $j$ we set 
\begin{align*}
\la_j^{(s)}=\left\{\begin{array}{ll}
\la_j^{(s-1)} & \quad\text{ if }1\leq j<t_s,\\
\la_j^{(s-1)}-1 & \quad\text{ if }t_s\leq j\leq 2i.
\end{array}\right.
\end{align*}

\item[Final setup] Assuming that the above program is terminated after $p\  (\leq 2i)$ steps, we obtain the partition pair $(\la^{(p)}, \mu^{(p)})$ where there is no odd part in $\la^{(p)}$.
\end{description}
Finally, we obtain that $\theta_1((\la, \mu))=(\beta, \gamma)=(\la^{(p)}, \mu^{(p)})\in \cF\times \cED$. And further we have $|\la|+|\mu|=|\beta|+|\ga|$ and $\ell(\la)=\ell(\beta)$. 
\end{description}

It is easy to verify that this map $\theta_2$ is well-defined and further it is a bijection.




\end{proof}

\begin{remark}
The version of the generating function corresponding to this bijection $\theta_2$ can be found in \cite[Eq.~(5.52)]{WZ25}:
\begin{align*}
\sum_{i, j\geq 0}\frac{q^{4i^2+2ij+\frac{j^2}{2}-2i+\frac{j}{2}}}{(q; q)_{2i}(q; q)_j}&=(-q; q)_{\infty}(-q^2; q^2)_{\infty}.
\end{align*}
However, it is worth mentioning that this bijection $\theta_2$ actually has proved a more general result, that is, it will preserve some partition statistics.
\end{remark}

\section{Two bijections for Theorem \ref{thm:para_CY}}\label{sec:comb pfCY}

In this section, we will complete the proof of the Theorem \ref{thm:para_CY}. During this proof process, a partition statistic on ordinary partitions will appear, which we call the 0-sequence of odd length in order to align with the definition of ``$\sol$''. Given an ordinary partition $\la\in \cP$, a maximal string of parts where adjacent parts equal contained in $\la$ is a {\it $0$-sequence} of $\la$. For example, for $\la=1+1+2+2+2+4+7+7+7+7\in \cP$, we know its all $0$-sequence is
\begin{align*}
1+1,\ 2+2+2,\ 4,\text{ and }7+7+7+7.
\end{align*}
First let us focus on the right hand sides of the two identities in Theorem \ref{thm:para_CY}. According to the combinatorial interpretation of Rogers-Ramanujan identity, we can easily see that their right hand sides generate the following partition sets (of course, a bijection established by the combinatorial framework ``base+increments'' can also yield this), respectively. 
\begin{align*}
\cRR^{\I}_n&=\{\la\in \cRR: \la=(4+16s_1)+(12+16s_2)+\cdots +(8n-4+16s_n)\},\\
\cRR^{\II}_n&=\{\la\in \cRR: \la=(8+8s_1)+(24+8s_2)+\cdots +(16n-8+8s_n)\},
\end{align*}
where the partition $S=(s_1, s_2, ..., s_n)\in \cP_n^*$. Let $\cP^*$ be the set of ordinary partitions where zeros can be considered as parts, and set $\cP^*_n:=\{\la\in \cP^*: \ell(\la)=n\}$. Note that zeros will also be considered a certain $0$-sequence. Moreover, let $\cRR^{\I}=\bigcup_{n\geq 1}\cRR^{\I}_n$ and $\cRR^{\II}=\bigcup_{n\geq 1}\cRR^{\II}_n$, then we have
\begin{align*}
\sum_{\la\in \cRR^{\I}}x^{2\ell(\la)}q^{|\la|}=\sum_{n\geq 0}\frac{x^{2n}q^{4n^2}}{(q^{16}; q^{16})_{n}},\text{ and }\sum_{\la\in \cRR^{\II}}x^{2\ell(\la)}q^{|\la|}=\sum_{n\geq 0}\frac{x^{2n}q^{8n^2}}{(q^{8}; q^{8})_{n}}.
\end{align*}
With these two sets in place, the right hand sides of Theorem \ref{thm:para_CY} is clear. Now our task is mainly focused on the left hand sides. Let us first look at the first identity, that is, identity \eqref{id:para_CY1}.  Note the following decomposition:
\begin{align*}
\text{LHS}=\sum_{i, j\geq 0}\frac{(-1)^jx^{i+j}q^{i^2+2ij+j^2}}{(q^8; q^8)_i(q^8; q^8)_j}=\sum_{i\geq 0}\frac{x^iq^{1+3+\cdots +(2i-1)}}{(q^8; q^8)_i}\cdot\sum_{j\geq 0}\frac{(-x)^jq^{(2i+1)+(2i+3)+\cdots +(2i+2j-1)}}{(q^8; q^8)_j}.
\end{align*}
Then we can interpret it as the generating function of the following set of (weighted) partition pair $(\la, \mu)\in \cA^{\VI}_i\times\cB^{\VI}_{i, j}$ for $i, j\geq 0$.
\begin{itemize}
\item[$\mathbf{(AVI)}$] Let $\cA_i^{\VI}$ be the set of all partitions with the following form:
\begin{align*}
\la=(1+8k_1)+ (3+8k_2)+\cdots + (2i-1+8k_i)
\end{align*}
where $K=(k_1, k_2, ..., k_i)\in \cP^*_i$ and each part of $\la$ is labeled as $x$.

\item[$\mathbf{(BVI)}$] Let $\cB_{i, j}^{\VI}$ be the set of all partitions with the following form:
\begin{align*}
\mu=(2i+1+8t_1)+(2i+3+8t_2)+(2i+2j-1+8t_j)
\end{align*}
where $T=(t_1, t_2, ..., t_j)\in \cP^*_j$ and each part of $\mu$ is labeled as $(-x)$.
\end{itemize} 
Next we may use these partition sets to construct the involution $\psi_1$ on $\bigcup_{i, j\geq 0}(\cA^{\VI}_i\times \cB^{\VI}_{i, j})$.

\begin{lemma}\label{lem:inv_CY1}
There exists an involution
\begin{align*}
\psi_1: \bigcup_{i, j\geq 0}(\cA^{\VI}_i\times \cB^{\VI}_{i, j})&\ri\bigcup_{i, j\geq 0}(\cA^{\VI}_i\times \cB^{\VI}_{i, j})\\
(\la, \mu)&\mapsto (\beta, \gamma),
\end{align*}
such that $|\la|+|\mu|=|\beta|+|\gamma|$ and $\ell(\la)+\ell(\mu)=\ell(\beta)+\ell(\ga)$. And when $(\la,\mu)\neq (\beta, \gamma)$, they have the same weights and the opposite signs. Consequently, the identity \eqref{id:para_CY1} follows.
\end{lemma}

\begin{proof}
Firstly we describe the map $\psi_1$ with help of $2$L-shape. Given a pair of partitions $(\la, \mu)\in \cA^{\VI}_i\times \cB^{\VI}_{i, j}$, i.e., 
\begin{align*}
\la=(1+8k_1)+\cdots +(2i-1+8k_i),\text{ and }\mu=(2i+1+8t_1)+\cdots + (2i+2j-1+8t_j),
\end{align*}
then we first need an important notion, $\f(\la)$, the first (the position from left to right) part in the first $0$-sequence of odd length in $K=(k_1, k_2, ..., k_i)$. More precisely, if $K$ has no $0$-sequence of odd length, we set $\f(\la):=+\infty$. Furthermore, we set $t_{1}:=+\infty$ if $\mu$ is the empty partition $\ep=(0, ..., 0)$. Then there are following two cases to consider.

\begin{description}
\item[CASE I] If $\f(\la)\leq t_{1}$ and suppose $\f(\la)=k_m$ for some $1\leq m\leq i$, then we remove the 2L-shape at $(2m-1+8_m)$ and add the size $s_{m, 2}(\la)=2i-1+8k_m$ into $\mu$ as the new smallest part. That is, 
\begin{align*}
\la&\mapsto (1+8k_1)+\cdots +(2m-3+8k_{m-1})+(2m-1+8k_{m+1})+\cdots + (2i-3+8k_i)=\beta,\\
\mu&\mapsto(2i-1+8k_m)+(2i+1+8t_1)+(2i+3+8t_2)+\cdots +(2i+2j-1+8t_j)=\ga.
\end{align*}

\item[CASE II] If $\f(\la)>t_1$, then remove $(2i+1+8t_1)$ from $\mu$ and insert it into $\la$ as a 2L-shape. Hence, we obtain $\ga=(2i+3+8t_2)+\cdots +(2i+2j-1+8t_j)$. And we will now present two cases for this unique insertion way: (i) if $k_1\geq t_1$, then we set $\beta=(1+8t_1)+(3+8k_1)+\cdots +(2i+1+8k_i)$; (ii) otherwise, there must be a $p$ such that $k_p<t_1$. Then we obtain $\beta=(1+8k_1)+\cdots +(2p-1+k_p)+(2p+1+t_1)+(2p+3+t_{p+1})+\cdots +(2i+1+8k_i)$.

\end{description}
Subsequently, we will check next items. 
\begin{itemize}
\item[(1)] $(\beta, \gamma)$ in both two cases are well-defined. In fact, from the construction of $\psi_1$ we easily see that $(\beta, \gamma)\in \cA^{\VI}_{i-1}\times \cB_{i-1, j+1}^{\VI}$ in case I, and $(\beta, \gamma)\in \cA^{\VI}_{i+1}\times \cB_{i+1, j-1}^{\VI}$.

\item[(2)] $\psi_1$ is an involution, that is, $\psi_1^2((\la, \mu))=(\la, \mu)$ for all pairs $(\la, \mu)\in \bigcup_{i, j\geq 0}(\cA^{\VI}_i\times \cB^{\VI}_{i, j})$. This is because the image of $(\la, \mu)$ in case I is in case II and the image of $(\la, \mu)$ in case II is in case I. Moreover, we have $|\la|+|\mu|=|\beta|+|\gamma|$ and when $(\la,\mu)\neq (\beta, \gamma)$ they have the same weights and the opposite signs.
\end{itemize}

Finally, if $\f(\la)=t_1=+\infty$, that means there is no $0$-sequence of odd length in $K=(k_1, ..., k_i)$ and $\cB^{\VI}_{i, j}=\emptyset$, then we obtain the set of fixed points omitting the empty partition:
\begin{align*}
\cM_i^{\I}=\{\la\in \cRR: \la=(1+8k_1)+(3+8k_2)+\cdots +(2i-1+8k_i),\text{ and }(k_1, ..., k_i)\in \cP^*_i\}, \end{align*}
where $i$ is even. Now we let $i=2n$ for non-negative integer $n$, then combine two adjacent parts into one part starting from the smallest part:
\begin{align*}
\pi=(4+16s_1)+(12+16s_2)+\cdots +(8n-4+16s_n)\in \cRR^{\I}_n,
\end{align*}
where $s_m=k_{2m-1}+k_{2m}$ for all $1\leq m\leq n$. It is easy to see that the map
\begin{align*}
\cM_{2n}^{\I}&\ri \cRR_{n}^{\I}\\
\la&\mapsto\pi
\end{align*}
is a bijection for all $n\geq 0$, then we complete the proof of \eqref{id:para_CY1}.
\end{proof}
 
In the next part of this section, we will prove the second identity in Theorem \ref{thm:para_CY}, that is, identity \eqref{id:para_CY2}. Note that we have the following decomposition for its left hand side:
\begin{align*}
 \text{LHS}=\sum_{i, j\geq 0}\frac{(-1)^{\binom{i-j}{2}}x^{i+j}q^{3i^2+2ij+3j^2}}{(q^4; q^4)_i(q^4; q^4)_j}=\sum_{i, j\geq 0}(-1)^{\binom{i-j}{2}}q^{2ij}\cdot \frac{x^iq^{3+9+\cdots +(6i-3)}}{(q^4; q^4)_i}\cdot \frac{x^jq^{3+9+\cdots +(6j-3)}}{(q^4; q^4)_j}.
\end{align*}
Then we can interpret it as the generating function of the following set of (weighted) partition triple $(\la, \mu, \eta^{(i, j)})\in \cA_i^{\VII}\times \cB_j^{\VII}\times \{\eta^{(i, j)}\}$ for $i, j\geq 0$.

\begin{itemize}
\item[$\mathbf{(AVII)}$] Let $\cA_i^{\VII}$ be the set of partitions with the following form:
\begin{align*}
\la=(3+4a_1)+(9+4a_2)+\cdots +(6i-3+4a_i),
\end{align*} 
where $A=(a_1, a_2, ..., a_i)\in \cP^*_i$ and each part of $\la$ is labeled as $x$.

\item[$\mathbf{(BVII)}$] Let $\cB_j^{\VII}$ be the set of partitions with the following form:
\begin{align*}
\mu =(3+4b_1)+(9+4b_2)+\cdots +(6j-3+4b_j),
\end{align*}
where $B=(b_1, b_2, ..., b_j)\in \cP^*_j$ and each part of $\mu$ is labeled as $x$.

\item[$\mathbf{(ETA)}$] The last set $\{\eta^{(i, j)}\}$ consists of the following partitions:
\begin{align*}
\eta^{(i, j)}=(2, 2, ..., 2)_{ij}
\end{align*}
where there are $(ij)$ twos and each $\eta^{(i, j)}$ is assigned with a sign $(-1)^{\binom{i-j}{2}}$.
\end{itemize}
Note that 
\begin{align*}
\binom{i-j}{2}\text{ is }\left\{\begin{array}{lllll}
\text{odd},\ &\text{ if }i-j\equiv 2, 3\ (\mathrm{mod}\ 4),\\
\text{even},\ &\text{ if }i-j\equiv 0, 1\ (\mathrm{mod}\ 4).
\end{array}\right.
\end{align*}
It means by simultaneously adding one to $i$ and subtracting one from $j$ (or adding one to $j$ and subtracting one from $i$), we can obtain the opposite parity compared to that before the transformation. This is the original idea behind constructing the involution. Next we will provide the involution.

\begin{lemma}\label{lem:bij_CY2}
There exists an involution
\begin{align*}
\psi_2: \bigcup_{i, j\geq 0}(\cA_i^{\VII}\times \cB_j^{\VII}\times \{\eta^{(i, j)}\})&\ri\bigcup_{i, j\geq 0}(\cA_i^{\VII}\times \cB_j^{\VII}\times \{\eta^{(i, j)}\})\\
(\la, \mu, \nu)&\mapsto(\beta, \gamma, \pi),
\end{align*}
such that $|\la|+|\mu|+|\nu|=|\beta|+|\gamma|+|\pi|$ and $\ell(\la)+\ell(\mu)=\ell(\beta)+\ell(\ga)$. And when $(\la, \mu, \nu)\neq (\beta, \gamma, \pi)$, they have the same weights and the opposite signs. Consequently, the identity \eqref{id:para_CY2} follows.
\end{lemma}

\begin{proof}
For given a partition triple $(\la, \mu, \nu)\in \cA_i^{\VII}\times \cB_{j}^{\VII}\times \{\eta^{(i, j)}\}$ where $\la$ and $\mu$ are defined above and $\nu=\eta^{(i, j)}$, firstly we need to make a convention, that is, $a_{i+1}=+\infty$ and $b_{j+1}=+\infty$. Now suppose that $p$ is the position where the first time a different part appears in these two partitions $\la$ and $\mu$. That is, $a_1=b_1, ..., a_{p-1}=b_{p-1}$ and $a_p\neq b_p$. There exist two cases as follows.
\begin{description}
\item[CASE I] If $a_p>b_p$, then we perform the following operations:
\begin{align*}
\la&\mapsto \beta=\left\{\begin{array}{lll}(3+4a_1)+ \cdots +(6p-9+4a_{p-1})\\
+(6p-3+4b_p)+(6p+3+4(a_p-1))+ \cdots +( 6i+3+4(a_i-1))\end{array}\right\};\\
\mu&\mapsto \gamma=\left\{\begin{array}{llll}(3+4b_1)+\cdots +( 6p-9+4b_{p-1})\\
+ (6p-3+4(b_{p+1}+1))+\cdots +(6j-9+4(b_j+1))
\end{array}\right\};\\
\nu & \mapsto \pi=(2, 2, ..., 2)_{(i+1)(j-1)}.
\end{align*}

\item[CASE II] If $a_p<b_p$, then we perform the following operations:
\begin{align*}
\la&\mapsto\beta=\left\{\begin{array}{lll}(3+4a_1)+\cdots +(6p-9+4a_{p-1})\\
+(6p-3+4(a_{p+1}+1))+ \cdots +(6i-9+4(a_i+1))
\end{array}\right\};\\ 
\mu&\mapsto \ga=\left\{\begin{array}{lll}(3+4b_1)+\cdots +(6p-9+4b_{p-1})\\
+(6p-3+4a_p)+(6p+3+4(a_{p}-1))+\cdots +(6j+3+4(b_j-1))
\end{array}\right\};\\
\nu&\mapsto \pi=(2, 2, ..., 2)_{(i-1)(j+1)}.
\end{align*}
\end{description}
Easily check that 
\begin{itemize}
\item[(1)] $(\beta, \gamma, \pi)$ in above two cases are well-defined. Since in case I, the resulting partition triple $(\beta, \gamma, \nu)\in \cA_{i+1}^{\VII}\times \cB_{j-1}^{\VII}\times\{\eta^{(i+1, j-1)}\}$ and in case II, the resulting partition triple $(\beta, \gamma, \eta)\in \cA_{i-1}^{\VII}\times \cB_{j+1}^{\VII}\times\{\eta^{(i-1, j+1)}\}$.

\item[(2)] $\psi_2$ is a weight-preserving and sign-opposite involution. Since we may see that $|\la|+|\mu|+|\nu|=|\beta|+|\ga|+|\pi|$ and $\ell(\la)+\ell(\mu)=\ell(\beta)+\ell(\ga)$, further $\nu$ and $\pi$ always have the opposite signs in both two cases. Moreover, since $a_{p+1}\geq a_p$ and $b_{p+1}\geq b_p$, then the image of the triple $(\la, \mu, \nu)$ in case I is in case II and vice versa.
\end{itemize}

Finally, we can obtain the set of fixed points satisfying that $i=j$ and $a_m=b_m$ for all $1\leq m\leq i$. Then for $i\geq 0$ the set $\cM^{\II}_i$ of fixed points consists of all triples $(\la, \mu, \nu)$ with the following form: 
\begin{align*}
\nu&=\eta^{(i, i)}=(2, 6, ..., 4i-2);\\
\la=\mu&=(3+4a_1)+(9+4a_2)+\cdots +(6i-3+4a_i).
\end{align*}
Then the sum of the corresponding parts of these three partitions is
\begin{align*}
\varpi=(8+8a_1, 24+8a_2, ..., 16i-8+8a_i)\in \cRR^{\II}_i,
\end{align*}
where $(a_1, a_2, ..., a_i)\in \cP^*_i$. It is obvious that this map
\begin{align*}
\cM^{\II}_i&\ri \cRR_i^{\II}\\
(\la, \mu, \nu)&\mapsto\varpi
\end{align*}
is a bijection, thus we have completed the proof of the identity \eqref{id:para_CY2}.
\end{proof}

\section{Conclusion}\label{sec:conclusion}

In this paper, we have completed the bijective proofs of some parameterized Rogers-Ramanujan type identities. By connecting with the previous work \cite{FL241, FL242, FL243}, it can be observed that the techniques used in these bijections may be applied to prove more similar identities. One of the directions we will consider next is the combinatorial interpretations and proofs of multi-sum Rogers-Ramanujan type identities, exploring whether these combinatorial frameworks and bijections have more general applications. For example, In the work of Wang et al. \cite{WW25} they established many interesting $q$-series identities. In our next paper, we will provide combinatorial interpretations of both sides of these multi-sum identities, as well as bijections of these identities. On the other hand, there may also be more space for development regarding the statistic ``sol'' in the subject.



\end{document}